\documentclass[11pt,english]{smfart}

\usepackage{graphicx}
\usepackage{a4wide}
\newtheorem{theorem}{Theorem}[section]
\newtheorem{proposition}[theorem]{Proposition}
\newtheorem{lemma}[theorem]{Lemma}
\newtheorem{corollary}[theorem]{Corollary}
\newtheorem{definition}[theorem]{Definition}

\newtheorem{fact}[theorem]{Fact}
\theoremstyle{definition}

\newtheorem*{theoremA}{Theorem A}
\newtheorem*{theoremARS}{Theorem ARS}
\newtheorem*{theoremHIS}{Theorem HIS}
\newtheorem*{theoremIR}{Theorem IR}
\newtheorem*{propositionBS}{Proposition BS}

\newcommand\decdot{\raisebox{0.1ex}{\textbf{.}}}
\newcommand{\leftexp}[2]{{\vphantom{#2}}^{#1}{#2}}
\newcommand{\biinf}{{}^{\scriptscriptstyle\bullet}}

\numberwithin{equation}{section}

\title{Rational numbers with purely periodic $\beta$-expansion} 


\begin{document}

\author[B. Adamczewski]{Boris Adamczewski}
\address{CNRS, Universit\'e de Lyon, Universit\'e Lyon 1 \\ 
Institut Camille Jordan \\ 
43 boulevard du 11 novembre 1918  \\
69622 Villeurbanne Cedex France}
\email{Boris.Adamczewski@math.univ-lyon1.fr}

\author[Ch. Frougny]{Christiane Frougny}
\address{Universit\'e Paris 8 and \\
LIAFA UMR 7089, case 7014 \\
75205 Paris Cedex 13 France}
\email{Christiane.Frougny@liafa.jussieu.fr}

\author[A. Siegel]{Anne Siegel}
\address{CNRS, Universit\'e de Rennes 1 \\ 
UMR 6074 / INRIA IRISA \\
Campus de Beaulieu \\	 
35042 Rennes Cedex France}
\email{anne.siegel@irisa.fr}

\author[W. Steiner]{Wolfgang Steiner}
\address{CNRS, Universit\'e Paris Diderot -- Paris 7 \\
LIAFA UMR 7089, case 7014 \\
75205 Paris Cedex 13 France}
\email{steiner@liafa.jussieu.fr}
\thanks{This work has been supported by the Agence Nationale de la Recherche, grant ANR--06--JCJC--0073 ``DyCoNum''.}

\begin{abstract}   
We study real numbers $\beta$ with the curious property that the $\beta$-expansion of all sufficiently small positive rational numbers is purely periodic.  
It is known that such real numbers have to be Pisot numbers which are units of the number field they generate. 
We complete known results due to Akiyama to characterize algebraic numbers of degree $3$ that enjoy this property.  
This extends results previously obtained in the case of degree $2$ by Schmidt, Hama and Imahashi.  
Let $\gamma(\beta)$ denote the supremum of the real numbers $c$ in $(0,1)$ such that all positive rational numbers less than $c$ have a purely periodic $\beta$-expansion.  
We prove that $\gamma(\beta)$ is irrational for a class of cubic Pisot units that contains the smallest Pisot number $\eta$. 
This result is motivated by the observation of Akiyama and Scheicher that $\gamma(\eta)=0.666\, 666 \, 666 \,086 \cdots$ is surprisingly close to $2/3$.
\end{abstract}

\maketitle

\section{Introduction}

One of the most basic results  about decimal expansions is that every rational number has an eventually periodic expansion\footnote{A sequence $(a_n)_{n\geq 1}$ is eventually periodic if there exists a positive integer $p$ such that 
$a_{n+p} = a_n$ for every positive integer $n$ large enough.}, the converse being obviously true.  
In fact, much more is known for we can easily distinguish rationals with a purely periodic expansion\footnote{A sequence $(a_n)_{n\geq 1}$ is purely periodic if there exists a positive integer $p$ such that $a_{n+p} = a_n$ for every positive integer $n$.}: a rational number $p/q$ in the interval $(0,1)$, in lowest form, has a purely periodic decimal expansion if and only if $q$ and $10$ are relatively prime. 
Thus, both rationals with a purely periodic expansion and rationals with a non-purely periodic expansion are, in some sense, uniformly spread on the unit interval. 
These results extend \emph{mutatis mutandis} to any integer base $b \geq 2$, as explained in the standard monograph of Hardy and Wright~\cite{Hardy&Wright}.

However, if one replaces the integer $b$ by an algebraic number that is not a rational integer, it may happen that the situation would be drastically different. 
As an illustration of this claim, let us consider the following two examples. 
First, let $\varphi$ denote the golden ratio, that is, the positive root of the polynomial $x^2-x-1$. 
Every real number $\xi$ in $(0,1)$ can be uniquely 
expanded as  
$$
\xi = \sum_{n\geq 1} \frac{a_n}{\varphi^n},
$$
where $a_n$ takes only the values $0$ and~$1$, and with the additional condition that $a_n a_{n+1}=0$ for every positive integer~$n$. 
The binary sequence $(a_n)_{n\geq 1}$ is termed the $\varphi$-expansion of~$\xi$. 
In 1980, Schmidt~\cite{Schmidt} proved the intriguing result that every rational number in $(0,1)$ has a purely periodic 
$\varphi$-expansion. 
Such a regularity is somewhat surprising as one may imagine $\varphi$-expansion of rationals more intricate than their decimal expansions. 
Furthermore, the latter property seems to be quite exceptional.    
Let us now consider $\theta=1+\varphi$, the largest root of the polynomial $x^2-3x+1$. 
Again, every real number $\xi$ in $(0,1)$ has a $\theta$-expansion, that is, $\xi$ can be uniquely expanded as  
$$
\xi= \sum_{n\geq 1} \frac{a_n}{\theta^n},
$$
where $a_n$ takes only the values $0$, $1$ and~$2$, (and with some extra conditions we do not care about here). 
In contrast to our first example, it was proved by Hama and Imahashi~\cite{HamaImahashi} that no rational number in $(0,1)$ has a purely periodic $\theta$-expansion. 

\medskip

Both $\varphi$- and $\theta$-expansions mentioned above are typical 
examples of the so-called \emph{$\beta$-expansions} introduced by R\'enyi~\cite{Renyi}.
Let $\beta>1$ be a real number. 
The $\beta$-expansion of a real number $\xi \in [0,1)$ is defined as the sequence $d_{\beta}(\xi) = (a_n)_{n\geq 1}$ over the alphabet 
$\mathcal{A}_{\beta}:=\{0,1,\ldots,\lceil\beta\rceil-1\}$ produced by the $\beta$-transformation $T_{\beta}:\,x\mapsto\beta x\bmod{1}$ with a greedy procedure; that is, such that, for all $i\geq 1$, $a_n=\lfloor\beta T_{\beta}^{n-1}(\xi)\rfloor$. 
The sequence $d_{\beta}(\xi)$ replaces in this framework the classical sequences of decimal and binary digits since we have 
$$
\xi=\sum_{n\geq 1} \frac{a_n}{\beta^n}\,.
$$ 

\medskip

Set 
$$ 
\gamma(\beta) := \sup\{c\in[0,1) \mid \forall\,0 \leq p/q \leq c, \, \,d_{\beta}(p/q)\ \mbox{is a purely periodic sequence}\}.
$$ 
This note is concerned with those real numbers $\beta$ with the property that all sufficiently small rational numbers have a purely periodic $\beta$-expansion, that is, such that 
\begin{equation}\label{gt0}
\gamma(\beta) > 0.
\end{equation} 
With this definition, we get that $\gamma(\varphi)=1$, while $\gamma(\theta)=0$. 
As one could expect, Condition~(\ref{gt0}) turns out to be very restrictive.  
We deduce from works of Akiyama~\cite{AKI2} and Schmidt~\cite{Schmidt80} (see details in Proposition~\ref{justify}) that such real numbers $\beta$ have to be Pisot units. 
This means that $\beta$ is both a Pisot number and a unit of the integer ring of the number field it generates. 
Recall that a Pisot number is a real algebraic integer which is greater than~$1$ and whose Galois conjugates (different from itself) are all inside the open unit disc. 

One relevant property for our study is
$$
{\rm (F)}: \quad \mbox{every $x\in\mathbb{Z}[1/\beta]\cap[0,1)$ has a finite $\beta$-expansion.}
$$ 
This property was introduced by Frougny and Solomyak in~\cite{FrougnySolomyak92}. 
It then has been studied for various reasons during the twenty last years.  
In particular, Akiyama~\cite{AKI1} proved the following unexpected result.

\begin{theoremA}
\emph{If $\beta$ is a Pisot unit satisfying {\rm (F)}, then $\gamma(\beta)>0$.}
\end{theoremA}

The fact that (F) plays a crucial role in the study of $\gamma(\beta)$ looks somewhat puzzling but it will become more transparent in the sequel.

\medskip

The results of Hama, Imahashi and Schmidt previously mentioned about $\varphi$- and $\theta$-expansions are actually more general and lead, using Proposition~\ref{justify}, to a complete understanding of $\gamma(\beta)$ when $\beta$ is a quadratic number. 

\begin{theoremHIS}
\emph{Let $\beta > 1$ be a quadratic number.
Then, $\gamma(\beta) > 0$ if and only if $\beta$ is a Pisot unit satisfying {\rm (F)}.
In that case, $\gamma(\beta) = 1$.}
\end{theoremHIS}

Furthermore, quadratic Pisot units satisfying (F) have been characterized in a simple way: they correspond to positive roots of polynomials $x^2-n x-1$, $n$ running along the positive integers. 
These exactly correspond to quadratic Pisot units whose Galois conjugate is negative.

\medskip
First, we establish the converse of Theorem~A for algebraic numbers of degree~3.
This provides a result similar to Theorem~HIS in that case.

\begin{theorem}\label{th1}
Let $\beta > 1$ be a cubic number.
Then, $\gamma(\beta) > 0$ if and only if $\beta$ is a Pisot unit satisfying {\rm (F)}.
\end{theorem}

We recall that Pisot units of degree~$3$ satisfying (F) have been nicely characterized in~\cite{Akiyama00}:  they correspond to 
the largest real roots of polynomials $x^3-ax^2-bx-1$, with $a,b$ integers, $a\geq 1$ and $-1\leq b \leq a+1$. 

\medskip

It is tempting to ask whether, in Theorem~\ref{th1}, Property (F) would imply that $\gamma(\beta)=1$ as it is the case for quadratic Pisot units.  
However, Akiyama \cite{AKI1} proved that such a result does not hold. 
Indeed, he obtained that the smallest Pisot number $\eta$, which is the real root of the polynomial $x^3-x-1$, satisfies $0< \gamma(\eta)<1$. 
More precisely, it was proved in \cite{AkiyamaSchei} that $\gamma(\eta)$ is abnormally close to the rational number $2/3$ since one has 
$$
\gamma(\eta) = 0.666 \, 666 \, 666 \, 086  \cdots.
$$
This intriguing phenomenon naturally leads to ask about the arithmetic nature of $\gamma(\eta)$.

\medskip

In this direction, we will prove that $\gamma(\eta)$ is irrational as a particular instance of the following result.  

\begin{theorem}\label{th2}
Let $\beta$ be a cubic Pisot unit satisfying {\rm (F)} and such that the number field $\mathbb{Q}(\beta)$ is not totally real.  
Then, $\gamma(\beta)$ is irrational. 
In particular, $0<\gamma(\beta)<1$.
\end{theorem}

Note that in Theorem~\ref{th2} the condition that $\beta$ does not generate a totally real number field is equivalent to the fact that the Galois conjugates of $\beta$ are complex (that is, they belong to $\mathbb{C}\setminus \mathbb{R}$). 
In all this note, a complex Galois conjugate of an algebraic number $\beta$ denotes a Galois conjugate that belongs to $\mathbb{C}\setminus \mathbb{R}$.

\medskip
 
The proofs of our results rely on some topological properties of the tiles of the so-called Thurston tilings associated with Pisot units. 
We introduce the notion of spiral points for compact subsets of $\mathbb{C}$ that turns out to be crucial for our study. 
The fact that $\gamma(\beta)$ vanishes for a cubic Pisot unit that does not satisfy (F) is a consequence of the fact that the origin is a spiral point with respect to the central tile of the underlying tiling. 
Theorem~\ref{th2} comes from the fact that $\gamma(\beta)$ cannot be a spiral point with respect to this tile, which provides a quite unusual proof of irrationality.  


\section{Expansions in a non-integer base}

In this section, we recall some classical results and notation about $\beta$-expansions. 
We first  explain why we will focus on bases $\beta$ that are Pisot units.

\begin{proposition}\label{justify}
Let $\beta>1$ be a real number that is not a Pisot unit. 
Then, $\gamma(\beta)=0$.
\end{proposition}

\begin{proof}
Assume that $\gamma(\beta)>0$. 
If the $\beta$-expansion of a rational number $r\in(0,1)$ has period~$p$, the following relation holds:
\begin{equation} \label{eq:1}
r=\frac{a_1\beta^{p-1}+a_2\beta^{p-2}+\cdots+a_p}{\beta^p-1}\,.
\end{equation}

Since $\gamma(\beta)>0$, $1/n$ has a purely periodic expansion for every positive integer $n$ large enough. 
We thus infer from (\ref{eq:1}) that $\beta$ is an algebraic integer.

A similar argument applies to prove that $\beta$ is a unit (see \cite[Proposition~6]{AKI1}).

It was proved in \cite{Schmidt80} that the set of rational numbers with a purely periodic $\beta$-expansion is nowhere dense in $(0,1)$ if $\beta$ is an algebraic integer that is not a Pisot neither a Salem number\footnote{An algebraic integer $\beta>1$ is a Salem number if all its Galois conjugates (different from~$\beta$) have modulus at most equal to $1$ and with at least one Galois conjugate of modulus equal to~$1$.}. 
Hence $\beta$ is either a Pisot unit or a Salem number.

Let us assume that $\beta$ is a Salem number. 
It is known that minimal polynomials of Salem numbers are reciprocal, which implies that $1/\beta \in (0,1)$ is a conjugate of $\beta$. 
Using Galois conjugaison, relation (\ref{eq:1}) still holds when replacing $\beta$ by~$1/\beta$. 
In this equation, the right hand-side then becomes non-positive. 
Therefore 0 is the only number with purely periodic $\beta$-expansion and $\gamma(\beta)=0$, a contradiction with our initial assumption. 
Note that this argument already appeared in \cite[Proposition~5]{AKI1}. 

Thus $\beta$ is a Pisot unit, which concludes the proof.
\end{proof}

In the sequel, we consider finite, right infinite, left infinite and bi-infinite words over the alphabet $\mathcal{A}_{\beta} := \{0,1,\ldots,\lceil\beta\rceil-1\}$. 
When the location of the 0 index is needed in a bi-infinite word, it is denoted by the $\biinf$ symbol, as in $\cdots a_{-1} a_0\biinf a_1 a_2\cdots$.
A~suffix of a right infinite word $a_0a_1\cdots$ is a right infinite word of the form $a_k a_{k+1}\cdots$ for some non-negative integer~$k$. 
A~suffix of a left infinite word $\cdots a_{-1}a_0$ is a finite word of the form $a_{k}a_{k+1}\cdots a_0$ for some non-positive integer~$k$.
A~suffix of a bi-infinite word $\cdots a_{-1} a_0\biinf a_1\cdots$ is 
a right infinite word of the form $a_{k}a_{k+1}\cdots$ for some integer~$k$. 
Given a finite word $u=u_0\cdots u_r$, we denote by $u^{\omega}=u_0\cdots u_r u_0\cdots u_r\cdots$ 
(resp.\ $\leftexp{\omega}{u} = \cdots u_0\cdots u_r u_0\cdots u_r$) the right (resp.\ left) infinite periodic word obtained by an infinite concatenation of~$u$. 
A~left infinite word $\cdots a_{-1}a_0$ is eventually periodic if there exists a positive integer $p$ such that $a_{-n-p} = a_{-n}$ for every positive integer $n$ large enough.

\medskip

It is well-known that the $\beta$-expansion of~$1$ plays a crucial role. 
Set $d_{\beta}(1):=(t_i)_{i\geq1}$. 
When $d_{\beta}(1)$ is finite with length~$n$, that is when $t_n\neq0$ and $t_i=0$ for every $i>n$, an infinite expansion of $1$ is given by $d_{\beta}^*(1)=(t_1\cdots t_{n-1}(t_n-1))^\omega$.
If $d_{\beta}(1)$ is infinite we just set $d_{\beta}^*(1)=d_{\beta}(1)$. The knowledge of this improper expansion of~$1$ allows to decide whether a given word over $\mathcal{A}_{\beta}$ is the $\beta$-expansion of some real number.   

\begin{definition}
A~finite, left infinite, right infinite or bi-infinite word over the alphabet $\mathcal{A}_\beta$ is an \emph{admissible} word if all its suffixes are lexicographically smaller than~$d_{\beta}^*(1)$.
\end{definition}

A~classical result of Parry~\cite{Parry60} is that a finite or right infinite word $a_1 a_2\cdots$ is the $\beta$-expansion of a real number in $[0,1)$ if and only if it is admissible. 
Admissible conditions are of course easier to check when $d_{\beta}^*(1)$ is eventually periodic. 
In the case where $\beta$ is a Pisot number, $d_{\beta}^*(1)$ is eventually periodic and every element in $\mathbb{Q}(\beta)\cap[0,1)$ has eventually periodic $\beta$-expansion according to~\cite{BM77,Schmidt80}. 
In contrast, algebraic numbers that do not belong to the number field $\mathbb{Q}(\beta)$ are expected to have a chaotic $\beta$-expansion (see~\cite{AdBu}). 
Note that if the set of real numbers with an eventually periodic $\beta$-expansion forms a field, then $\beta$ is either a Salem or a Pisot number~\cite{Schmidt80}.

\medskip

\noindent{\bf \itshape Warning.} --- In all what follows, $\beta$ will denote a Pisot unit and admissibility will refer to this particular Pisot number. 
The sequence $d_{\beta}^*(1)$ is thus eventually periodic and we set 
$d_{\beta}^*(1) := t_1 t_2\cdots t_m(t_{m+1}\cdots t_{m+n})^{\omega}$. 
Notice that since $\beta$ is a Pisot unit, we have 
$\mathbb{Z}[1/\beta] = \mathbb{Z}[\beta]$.

\section{Thurston's tiling associated with a Pisot unit}

Given a real number $\beta$ there is a natural way to tile the real line using the notion of $\beta$-integers (see below). 
In his famous lectures, Thurston~\cite{thurston} discussed the construction of a dual tiling, a sort of Galois conjugate of this tiling, when $\beta$ is a Pisot number. 
Note that some examples of similar tiles had previously been introduced by Rauzy~\cite{rau1} to study arithmetic properties of an irrational translation on a two-dimensional torus. 
In this section, we recall Thurston's construction. 

The vectorial space $\mathbb{R}^n \times \mathbb{C}^m$ is endowed with its natural product topology. 
In the sequel, the closure~$\overline{X}$, the interior $\mathring{X}$ and the boundary $\partial X$ of a subset $X$ of $\mathbb{R}^n \times \mathbb{C}^m$ will refer to this topology.

\medskip

The $\beta$-transformation induces a decomposition of every positive real number in a $\beta$-fractional and a $\beta$-integral part as follows:
Let $k\in\mathbb{N}$ be such that $\beta^{-k}x\in[0,1)$ and $d_\beta(\beta^{-k}x)=a_{-k+1}a_{-k+2}\cdots$.
Then 
$$
x=\underbrace{a_{-k+1}\beta^{k-1}+\cdots+a_{-1}\beta+a_0}_{\mbox{$\beta$-integral part}} + \underbrace{a_1\beta^{-1}+a_2\beta^{-2}+\cdots}_{\mbox{$\beta$-fractional part}}. 
$$

In the following, we will use the notation $x=a_{-k+1}\cdots a_{-1}a_0\decdot a_1 a_2\cdots$.
We also note $x=a_{-k+1}\cdots a_{-1}a_0\decdot$ when the $\beta$-fractional part vanishes and $x=\decdot a_1 a_2\cdots$ when the $\beta$-integral part vanishes. 

\medskip

Since $a_{-k+1}=0$ if $x<\beta^{k-1}$, this decomposition does not depend on the choice of $k$. 
The set of $\beta$-integers is the set of positive real numbers with vanishing $\beta$-fractional part:
$$
\mathrm{Int}(\beta) := \left\{a_{-k+1}\cdots a_{-1}a_0\decdot \mid 
a_{-k+1}\cdots a_{-1}a_0\ \mbox{is admissible,}\ k\in\mathbb{N} \right\}.
$$

Let $\sigma_2,\ldots,\sigma_r$ be the non-identical real embeddings of $\mathbb{Q}(\beta)$ in $\mathbb{C}$ and let $\sigma_{r+1},\ldots, \sigma_{r+s}$ be the complex embeddings of $\mathbb{Q}(\beta)$ 
in~$\mathbb{C}$. 
We define the map $\Xi$ by:
\begin{eqnarray*}
\Xi:\ \mathbb{Q}(\beta) & \to & \mathbb{R}^{r-1}\times\mathbb{C}^s \\
x & \mapsto & (\sigma_2(x),\ldots,\sigma_{r+s}(x)).
\end{eqnarray*} 

\begin{definition} 
Let $\beta$ be a Pisot number. 
The compact subset of $\mathbb{R}^{r-1} \times \mathbb{C}^s$ defined by  
$$
\mathcal{T} := \overline{\Xi(\mathrm{Int}(\beta))}
$$
is called the central tile associated with~$\beta$. 
\end{definition}

\begin{figure}[ht]
\centerline{\includegraphics{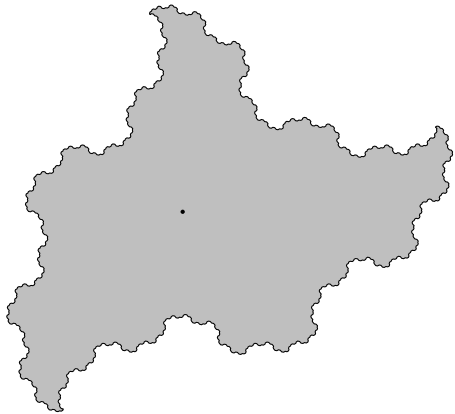}}
\caption{The central tile for the smallest Pisot number~$\eta$, which satisfies $\eta^3=\eta+1$.}
\end{figure}

The previous construction can be naturally extended to associate a similar tile with any $y\in\mathbb{Z}[\beta]\cap[0,1)$.  

\begin{definition} 
Given such a number $y$, we define the tile
\begin{equation*}
\mathcal{T}(y) :=\Xi(y)+\overline{\{\Xi(a_{-k+1}\cdots a_0\decdot)  \mid a_{-k+1}\cdots a_0 d_\beta(y)\ \mbox{is an admissible word}\}}. 
\end{equation*}
\end{definition} 
Note that $\mathcal{T}(0)=\mathcal{T}$. 
It is also worth mentioning that: 
\begin{itemize}
\itemsep5pt
\item[$\bullet$] 
there are exactly $n+m$ different tiles up to translation; 
\item[$\bullet$] 
the tiles $\mathcal{T}(y)$ induce a covering of the space $\mathbb{R}^{r-1}\times\mathbb{C}^s$, that is,   
\begin{equation}\label{eq:covering}
\bigcup_{y\in\mathbb{Z}[\beta]\cap[0,1)}\mathcal{T}(y) = \mathbb{R}^{r-1}\times\mathbb{C}^s.
\end{equation}
\end{itemize}

\begin{figure}[ht]
\centerline{\includegraphics[scale=.75]{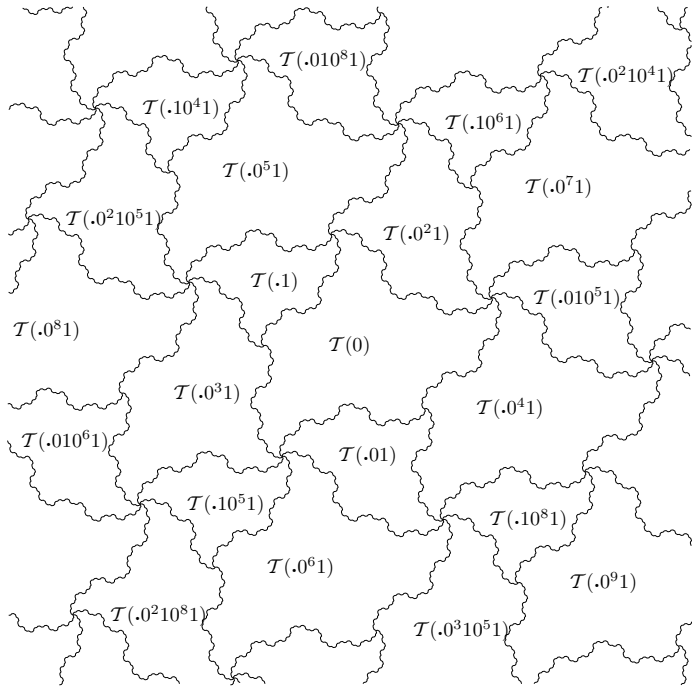}}
\caption{Aperiodic tiling associated with the smallest Pisot number~$\eta$.}
\end{figure}

The first observation follows from the fact that $d_{\beta}^*(1)$ is eventually periodic (see for instance~\cite{AKI2}). 
The second property is a consequence of the fact that $\Xi(\mathbb{Z}[\beta]\cap[0,\infty))$ is dense in $\mathbb{R}^{r-1}\times\mathbb{C}^s$, as proved in~\cite{AKI6,AKI2}.  

Furthermore, in the case where $\beta$ is a cubic Pisot unit, Akiyama, Rao and Steiner~\cite{AkiyamaRaoSteiner} proved that this covering is actually a tiling, meaning that the interior of tiles never meet and their boundaries have zero Lebesgue measure. 
We then have the following property.

\begin{theoremARS}
Let $\beta$ be a cubic Pisot unit. 
If $x$ and $y$ are two distinct elements in $\mathbb{Z}[\beta]\cap[0,1)$, then 
$$
\mathring{\mathcal{T}(x)} \cap \mathring{\mathcal{T}(y)} = \emptyset.
$$
\end{theoremARS}

In the sequel, we will also need the following observation.

\begin{fact} \label{fact}
There exists a constant $C$ such that every $z\in\mathbb{R}^{r-1}\times\mathbb{C}^s$ is contained in at most $C$ different tiles $\mathcal{T}(y)$.
\end{fact}

Indeed, since $\mathcal{T}(y)\subseteq \Xi(y)+\mathcal{T}$ and $\mathcal{T}$ is compact, a point $z\in\mathbb{R}^{r-1}\times\mathbb{C}^s$ cannot belong to the tile $\mathcal{T}(y)$ as soon as the distance between $y$ and $z$ is large enough. 
The result then follows since the set $\Xi(\mathbb{Z}[\beta]\cap[0,1))$ is uniformly discrete.

\section{A Galois type theorem for expansions in a Pisot unit base}

We introduce now a suitable subdivision of the central tile $\mathcal{T}$.  
The set $\mathrm{Int}(\beta)$ is a discrete subset of~$\mathbb{R}$, 
so that it can been ordered in a natural way. 
Then, we have the nice property that two consecutive points in $\mathrm{Int}(\beta)$ can differ only by a finite number of values. 
Namely, if $t_1\cdots t_i$, $i\ge0$, is the longest prefix of $d_\beta^*(1)$ which is a suffix of $a_{-k+1}\cdots a_0$, then this difference is equal to $T_{\beta}^i(1)$ (see~\cite{Akiyama07,thurston}). 
Since $d_\beta^*(1)=t_1\cdots t_m(t_{m+1}\cdots t_{m+n})^{\omega}$, then we have $T_\beta^i(1)=T_\beta^{i+n}(1)$ for $i\ge m$, which confirms our claim. 

A~natural partition of $\mathrm{Int}(\beta)$ is now given by considering the distance between a point and its successor in~$\mathrm{Int}(\beta)$.

\begin{definition}
For every $0\le i<m+n$, we define the subtile $\mathcal{T}_i$ of $\mathcal{T}$ to be the closure of the set of those points $\Xi(a_{-k+1}\cdots a_0\decdot)$ such that the distance from $a_{-k+1}\cdots a_0\decdot$ to its successor in $\mathrm{Int}(\beta)$ is equal to~$T_\beta^i(1)$.
\end{definition}

\begin{figure}[ht]
\centerline{\includegraphics{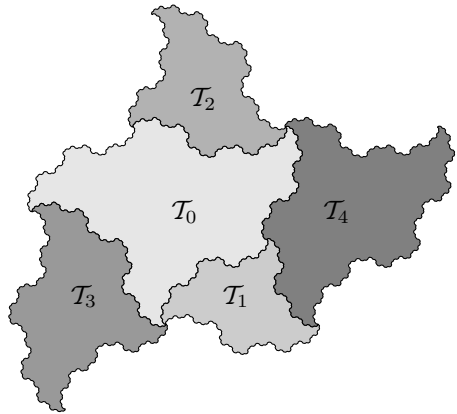}}
\caption{The decomposition of the central tile associated with $\eta$ into subtiles. 
According to the expansion $d_{\eta}(1)=10001$, the central tile is subdivided into exactly five subtiles.}
\end{figure}

As detailed in~\cite{BS2,BSJNT}, the Perron-Frobenius theorem coupled with self-affine decompositions of tiles implies that the subtiles have nice topological properties. 
One of them will be useful in the following: their interiors are disjoint.

\begin{propositionBS}
\emph{Let $\beta$ be a Pisot unit. 
For every pair $(i,j)$, $i\not=j$, $0\le i,j <m+n$,   
$$
\mathring{\mathcal{T}_i} \cap \mathring{\mathcal{T}_j}= \emptyset.
$$}
\end{propositionBS}

The subdivision of the central tile into the subtiles $\mathcal{T}_i$ allows to characterize those real numbers in $(0,1)$ having a purely periodic $\beta$-expansion. 
More precisely, Ito and Rao~\cite{Hui2} proved the following result (see also~\cite{BSJNT} for a shorter and more natural proof). 

\begin{theoremIR}
\emph{Let $\beta$ be a Pisot unit and $x\in[0,1)$. 
The $\beta$-expansion of $x$ is purely periodic if and only if $x\in\mathbb{Q}(\beta)$ and
$$
(-\Xi(x),x)\in \mathcal{E}_{\beta} := \bigcup_{i=0}^{n+m-1} \mathcal{T}_i\times[0,T_{\beta}^i(1)).
$$}
\end{theoremIR}

\begin{figure}[ht]
\centerline{\includegraphics[scale=.75]{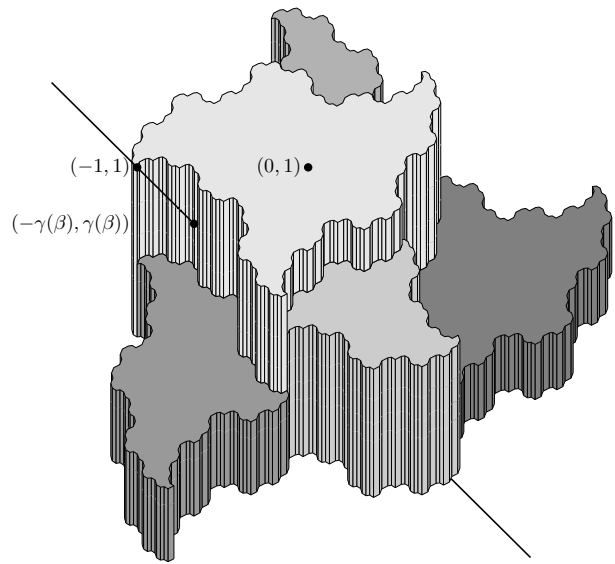}}
\caption{The set $\mathcal{E}_{\eta}$ and the line $(-x,x)$, $x \in \mathbb{R}$.}
\end{figure}

We immediately deduce the following result from Theorem IR.

\begin{corollary}\label{cor:boundary}
Let $\beta$ be a Pisot unit. Then, one of the following holds: 
\begin{itemize}
\itemsep5pt
\item[(i)]
$\gamma(\beta)=T_\beta^i(1)$ for some $i\in\{0,\ldots,n+m-1\}$,
\item[(ii)]
the $(r+s-1)$-dimensional vector $(-\gamma(\beta), \dots, -\gamma(\beta))$ is in $\mathcal{T}_i\cap\mathcal{T}_j$ with $T_{\beta}^j(1)<\gamma(\beta)<T_{\beta}^i(1)$, 
\item[(iii)]
the $(r+s-1)$-dimensional vector $(-\gamma(\beta), \dots, -\gamma(\beta))$ lies on the boundary of $\mathcal{T}$.
\end{itemize}
\end{corollary}

\section{Some results on $\Xi(\beta)$-representations}

In this section, we consider representations of points in the space $\mathbb{R}^{r-1} \times \mathbb{C}^s$ that involve some Galois conjugations related to~$\beta$. 
Such representations are termed $\Xi(\beta)$-representations. 
If $x=(x_1,\ldots,x_{r+s-1)}$ and $=(y_1,\ldots,y_{r+s-1})$ are two elements in $\mathbb{R}^{r-1} \times \mathbb{C}^s$, we set $x \odot z := (x_1 y_1,\ldots,x_{r+s-1}y_{r+s-1})$. 
From now on, we say that a pointed bi-infinite word $\cdots a_{-1}a_0\biinf a_1 a_2\cdots$ is admissible if $\cdots a_{-1}a_0 a_1 a_2\cdots$ is a bi-infinite admissible word.

\begin{definition}
A~\emph{$\Xi(\beta)$-representation} of $z\in\mathbb{R}^{r-1}\times\mathbb{C}^s$ is an admissible bi-infinite word $\cdots a_{-1}a_0\biinf a_1 a_2\cdots$ with $\decdot a_1 a_2\cdots\in\mathbb{Z}[\beta]$ and 
such that 
$$
z=\sum_{j=0}^\infty a_{-j}\Xi(\beta^j)+\Xi(\decdot a_1 a_2\cdots).
$$
\end{definition}

We derive now several results about $\Xi(\beta)$-representations that will be useful in the sequel. 
First, we show that such representations do exist.

\begin{lemma}\label{lem1}
Every $z\in\mathbb{R}^{r-1}\times\mathbb{C}^s$ has at least one $\Xi(\beta)$-representation.
\end{lemma}

\begin{proof}
As we already mentioned, we have 
$$
\bigcup_{y\in\mathbb{Z}[\beta]\cap[0,1)}\mathcal{T}(y) = \mathbb{R}^{r-1}\times\mathbb{C}^s.
$$
Thus every $z\in\mathbb{R}^{r-1}\times\mathbb{C}^s$ belongs to some tile $\mathcal{T}(y)$ with $y\in \mathbb{Z}[\beta]$. 
By the definition of $\mathcal{T}(y)$, there exists a sequence of finite words $W_k$ such that $W_k d_{\beta}(y)$ is an admissible infinite word and
$$
\lim_{k\to\infty} \Xi(W_k\decdot) = z - \Xi(y).
$$
Now, note that there exist infinitely many $W_k$ that end with the same letter, say~$a_0$. 
Among them, there are infinitely many of them with the same last but one letter, say~$a_{-1}$. 
Keeping on this procedure, we deduce the existence of a left  infinite word $\cdots a_{-1}a_0$ such that $\cdots a_{-1}a_0 \biinf  d_{\beta}(y)$ is a bi-infinite admissible word and 
$$
z - \Xi(y) = \lim_{k\to\infty} \Xi(a_{k+1}\cdots a_0\decdot).
$$
Thus, we have $z = \sum_{j=0}^{+\infty} a_{-j} \Xi(\beta^j)  + \Xi(y)$, which proved that $\cdots a_{-1}a_0\biinf d_{\beta}(y)$ is a $\Xi(\beta)$-representation of $z$ since $y$ belongs to $\mathbb{Z}[\beta]$. 
This ends the proof.
\end{proof} 

The following results were shown by Sadahiro~\cite{Sadahiro06} for cubic Pisot units $\beta$ satisfying (F) with a single pair of complex Galois conjugates, \emph{i.e.}\ $r=s=1$.

\begin{lemma}\label{lemC}
There exists a positive integer $C$ such that every $z\in\mathbb{R}^{r-1}\times\mathbb{C}^s$ has at most $C$ different $\Xi(\beta)$-representations.
\end{lemma}

\begin{proof} 
We already observed in Fact~\ref{fact} that there exists a positive integer, say $C$, such that every $z\in\mathbb{R}^{r-1}\times\mathbb{C}^s$ is contained in at most $C$ different tiles $\mathcal{T}(y)$. 
Let $z\in\mathbb{R}^{r-1}\times\mathbb{C}^s$ and let us assume that $z$ has more than $C$ different $\Xi(\beta)$-representations, namely
$$
(\cdots a_{-1}^{(j)}a_0^{(j)}\biinf a_1^{(j)}a_2^{(j)}\cdots)_{1\leq j \leq C+1}.
$$ 
Then, there exists some non-negative integer $k$ such that the infinite sequences $a_{-k+1}^{(j)}a_{-k+2}^{(j)}\cdots$, $1\leq j \leq C+1$, are all distinct. 
This implies that $\Xi(\beta^k)\odot z$ belongs to each tile $\mathcal{T}({y_j})$ with $y_j = \decdot a_{-k+1}^{(j)}a_{-k+2}^{(j)}\cdots$. 
Consequently, $\Xi(\beta^k)\odot z$ lies in more than $C$ different tiles, which contradicts the definition of $C$.
\end{proof}

\begin{lemma}\label{lemper}
Let $\cdots a_{-1}a_0\biinf a_1 a_2\cdots$ be a $\Xi(\beta)$-representation of $\Xi(x)$ for some $x\in\mathbb{Q}(\beta)$. 
Then, the left infinite word $\cdots a_{-1}a_0$ is eventually periodic. 
\end{lemma}

\begin{proof}
Let $\cdots a_{-1}a_0\biinf a_1 a_2\cdots$ be a $\Xi(\beta)$-representation of $\Xi(x)$ for some $x\in\mathbb{Q}(\beta)$. 
Set $x_k := a_{-k+1}\cdots a_0\decdot a_1 a_2\cdots$ and for every non-negative integer $k$ set $z_k :=\beta^{-k}(x-x_k)$. 
Then, we have 
$$
z_{k+1}=\frac{1}{\beta}(z_k-a_{-k}),
$$ 
which implies that the set $\{z_k \mid k\ge0\}$ is bounded.
Furthermore, we have 
$$
\Xi(z_k)=\sum_{j=0}^\infty a_{-k-j}\Xi(\beta^j).
$$ 
Since $\beta$ is a Pisot number, this implies that all conjugates of $z_k$ are bounded as well. 
Since all the $z_k$ have a degree bounded by the degree of $\beta$, this implies that $\{z_k \mid k\ge0\}$ is a finite set.

Now, observe that $\cdots a_{-k-1}a_{-k}\biinf 0^\omega$ is a $\Xi(\beta)$-representation of $\Xi(z_k)$. 
By Lemma~\ref{lemC}, each $\Xi(z_k)$ has at most $C$ different $\Xi(\beta)$-representations.  
Since there are only finitely many different $z_k$, the set of left infinite words $\{\cdots a_{-k-1}a_{-k} \mid  k\geq 0\}$ is finite.  
This implies that the left infinite word $\cdots a_{-1}a_0$ is eventually periodic, concluding the proof.
\end{proof}

\section{Spiral points}

In this section, we introduce a topological and geometrical notion  for compact subsets of the complex plane.  
This notion of \emph{spiral point} turns out to be the key tool 
for proving Theorems~\ref{th1} and~\ref{th2}. 
Roughly, $x$ is a spiral point with respect to a compact set $X\subset \mathbb{C}$ when both the interior and the complement of $X$ turn around~$x$, meaning that they meet
infinitely many times all rays of positive length issued from~$x$. More formally, we have the following definition.

\begin{definition}
Let $X$ be a compact subset of~$\mathbb{C}$. 
A point $z\in X$ is a \emph{spiral point} with respect to $X$ if for every positive real numbers $\varepsilon$ and $\theta$, both the interior of $X$ and the complement of $X$ meet the 
ray $z+[0,\varepsilon)e^{i\theta} := \left\{ z+\rho e^{i\theta} \mid \rho \in [0,\varepsilon)\right\}$.
\end{definition}

\begin{figure}[ht]
\centerline{\includegraphics{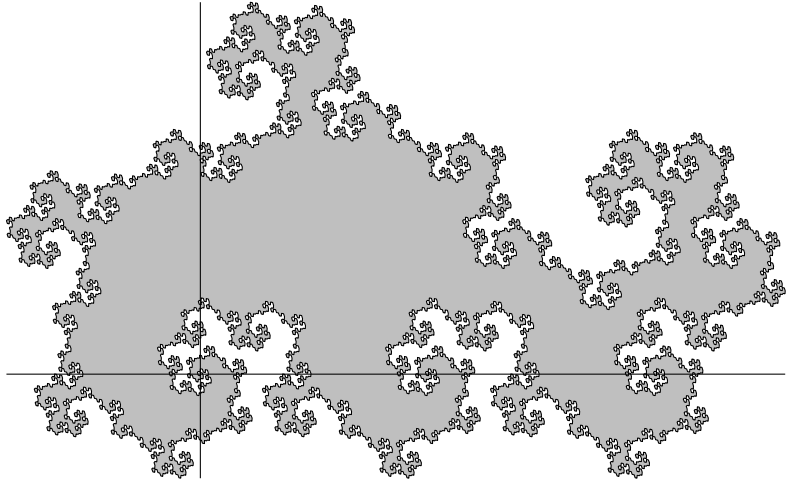}}
\caption{The central tile associated with the real root of $x^3-3x^2+2x-1$.}
\end{figure}

\medskip

It seems that the boundary of many fractal objects in the complex plane contains some spiral points, but we were not able to find a reference for this notion. 
The most common property studied in fractal geometry that is related to our notion of spiral point seems to be the non-existence of weak tangent (see for instance~\cite{Falconer}).  
For instance, if $X$ denotes a set with a non-integer Hausdorff dimension lying between $1$ and~$2$, then almost all points of $X$ do not have a weak tangent. 
This result applies in particular to the boundary of some classical fractal structures such as Julia sets and Heighway dragon.

\medskip

Our key result now reads as follows.

\begin{proposition}\label{prop:spiral}
Let $\beta$ be a cubic Pisot number with a complex Galois conjugate~$\alpha$.
Then every point in $\mathbb{Q}(\alpha)$ that belongs to   
the boundary of $\mathcal{T}$ or of a subtile $\mathcal{T}_i$ is a spiral point with respect to this tile.
\end{proposition}

In order to prove Proposition~\ref{prop:spiral} we will need the following result.

\begin{lemma}\label{lem:M}
Let $\beta$ be a Pisot number with complex Galois conjugates $\beta_j,\overline{\beta_j}$, $r<j\le r+s$, $\beta_j=\rho_j e^{2\pi i\phi_j}$.
Then $1,\phi_{r+1},\ldots,\phi_{r+s}$ are linearly independent over~$\mathbb{Q}$.
\end{lemma}

\begin{proof}
Let $k_0,\ldots,k_s$ be integers such that $k_0+k_1\phi_{r+1}+\cdots+k_s\phi_{r+s}=0$. 
Then, the product $\beta_{r+1}^{k_1}\cdots\beta_{r+s}^{k_s}$ is a real number and thus
$$
\beta_{r+1}^{k_1} \cdots \beta_{r+s}^{k_s} = \overline{\beta_{r+1}}^{k_1} \cdots \overline{\beta_{r+s}}^{k_s}.
$$
It is proved in \cite{Mignotte84} that this implies that $k_1=\cdots=k_s=0$, and the lemma is proved.
\end{proof}

We are now ready to prove Proposition~\ref{prop:spiral}.

\begin{proof}[Proof of Proposition~\ref{prop:spiral}] 
First note that since $\beta$ is a cubic Pisot number, we simply have $\Xi(\beta)=\alpha$.
Let $z\in\mathbb{Q}(\alpha)\cap\mathcal{T}$ and let $\theta$ and $\varepsilon$ be two positive real numbers. 

By Lemma~\ref{lem1}, $z$ has at least one $\alpha$-representation. 
Since $z$ belongs to the central tile~$\mathcal{T}$, the proof of Lemma~\ref{lem1} actually implies the existence of a left infinite 
admissible word $\cdots a_{-1}a_0$ such that $\cdots a_{-1}a_0\biinf 0^{\omega}$ is an $\alpha$-representation of~$z$. 
Furthermore, by Lemma~\ref{lemper}, such a representation is eventually periodic and there thus exists non-negative integers $p$ and $q$ such that $\leftexp{\omega}{(}a_{-q-p+1}\cdots a_{-q})a_{-q+1}\cdots a_0\biinf 0^{\omega}$ is an $\alpha$-representation of~$z$. 
 
For every non-negative integer~$k$, set $z_k :=(a_{-q-p+1}\cdots a_{-q})^k a_{-q+1}\cdots a_0\decdot$.
There exists a positive integer $\ell$ (depending on~$\beta$) such that for every left infinite admissible word $\cdots b_{-1}b_0$ the bi-infinite word  
\begin{equation}\label{equ:jump}
\cdots b_{-1}b_0 0^{\ell}(a_{-q-p+1}\cdots a_{-q})^j a_{-q+1}\cdots a_0 \biinf 0^{\omega}
\end{equation}
is also admissible. 
Roughly, this means that the lexicographic condition cannot ``jump over~$0^{\ell}$''. 
Consequently, we have 
\begin{equation}\label{equ:ball}
z_k+\alpha^{q+k p+\ell}\mathcal{T} \subseteq \mathcal{T}.
\end{equation}
Since $\beta$ is a Pisot unit, we know that $\mathcal{T}$ has a non-empty interior, a result obtained in~\cite{AKI2}.
Thus, $\mathcal{T}$ contains some ball, say $\mathcal{B}$. 
Set $\mathcal{B}':= \alpha^{q+\ell}\mathcal{B}$. 
By~(\ref{equ:ball}), it also contains the balls $z_k+\alpha^{k p}\mathcal{B}'$ for every non-negative integer~$k$.

Note that there exists some non-empty interval $(\eta,\zeta)\subset (0,1)$ and some positive real number $R$ such that every ray $z+[0,R)e^{2\pi i\psi}$ with $\psi\in(\eta,\zeta)$ contains an interior point of $z_0+\mathcal{B}'$. 
Furthermore, $\leftexp{\omega}{(}a_{-q-p+1}\cdots a_{-q})0^{q+k p} \biinf 0^{\omega}$ is an $\alpha$-representation of $z-z_k$ and thus 
$$
\alpha^{k p}\mathcal{B}'+z_k-z = \alpha^{k p}(\mathcal{B}'+z_0-z).
$$
Let $\rho$ and $\phi$ be positive real numbers such that 
$\alpha =\rho e^{2\pi i\phi}$.
Then, every ray 
$$
z+[0,\rho^{k p}R)e^{2\pi i(\psi+k p\phi)}
$$ 
with $\psi\in(\eta,\zeta)$ contains an interior point of $z_k+\alpha^{k p}\mathcal{B}'$.
By Lemma~\ref{lem:M}, $\phi$ is irrational and the sequence 
$(k p\phi \bmod{1})_{k\geq 0}$ is thus dense in $(0,1)$. 
It follows that there are infinitely many positive integers $k_1 < k_2 < \ldots$ and infinitely many real numbers $x_1,x_2,\ldots \in (\eta,\zeta)$ such that $\theta/2\pi = k_h p\phi+ x_h \bmod{1}$ for $h \ge 1$. 
Since $\beta$ is a Pisot number, we have $0<\rho <1$.  
For $\ell$ large enough, we thus obtain that the ray  $z+[0,\varepsilon)e^{i\theta}$ contains an interior point of~$\mathcal{T}$. 

If $z$ belongs to the boundary of~$\mathcal{T}$, from the covering  property (\ref{eq:covering}) it follows that $z$ is also contained in some tile $\mathcal{T}(y)$ with $y\ne 0$.

Then, arguing as previously, we obtain that $z$ has an $\alpha$-representation of the form 
$$
\leftexp{\omega}{(}a_{-q'-p'+1}'\cdots a_{-q'}')a_{-q'+1}'\cdots a_0'\biinf d_\beta(y)
$$
and by similar arguments as above, we can show that the ray $z+[0,\varepsilon)e^{i\theta}$ contains an interior point of the tile~$\mathcal{T}(y)$. 
By Theorem~ARS, such a point lies in the complement of~$\mathcal{T}$. 
This shows that $z$ is a spiral point with respect to~$\mathcal{T}$.

\medskip

We shall now detail why similar arguments apply if we replace $\mathcal{T}$ by a subtile~$\mathcal{T}_j$.
Recall that $d_\beta^*(1)=t_1\cdots t_m(t_{m+1}\cdots t_{m+n})^\omega$
denotes the expansion of~1. 
It follows from the definition of the subtiles $\mathcal{T}_j$ that:
\begin{itemize}
\itemsep3pt
\item[(i)] 
when $0 \leq j < m$,  a point $z$ belongs to $\mathcal{T}_j$ if and only if $z$ has an $\alpha$-representation with  $t_1 \cdots t_j\biinf 0^{\omega}$ as a suffix;
\item[(ii)] 
when $m \leq j < m+n$,  a point $z$ belongs to $\mathcal{T}_j$ if and only if there exists $\ell \geq 0$ such that $z$ has an $\alpha$-representation with  $t_1 \cdots t_m (t_{m+1}\cdots t_{m+n})^\ell t_{m+1} \cdots t_j \biinf 0^{\omega}$ as a suffix.
\end{itemize}

Let us assume that $z \in \mathbb{Q}(\alpha) \cap \partial \mathcal{T}_j$ for some $0 \leq j<m+n$. 
If $z$ also belongs to $\partial T$, we fall into the previous case. We can thus assume that this is not the case, hence $z$ belongs to the boundary of another tile $\mathcal{T}_h$, $h\neq j$.
By Proposition~BS, it remains to prove that the ray  $z+[0,\varepsilon)e^{i\theta}$ contains an interior point of both $\mathcal{T}_j$ and $\mathcal{T}_h$. 

Let us briefly justify this claim. 
Coming back to the proof above, we deduce from (i) and (ii) that $z_k$ belongs to $\mathcal{T}_j$ for $k$ large enough. 
Moreover, there exists a positive integer~$\ell$ (depending on~$\beta$) such that for every left infinite admissible word $\cdots b_{-1}b_0$ the bi-infinite word 
$$
\cdots b_{-1}b_0 0^{\ell}(a_{-q-p+1}\cdots a_{-q})^k a_{-q+1}\cdots a_0 \biinf 0^{\omega}
$$
satisfies the admissibility condition for ${\mathcal{T}}_j$ given in~(i)/(ii). 

This allows to replace (\ref{equ:ball}) by 
$$
z_k+\alpha^{q+k p+\ell}\mathcal{T} \subseteq \mathcal{T}_j
$$
and to conclude as previously.
\end{proof}

\section{Proof of Theorems~\ref{th1} and~\ref{th2}}

In this section, we complete the proof of Theorems~\ref{th1} and~\ref{th2}.  
In this section, $\beta$ denotes a cubic Pisot unit. 

\begin{lemma}\label{lem:Ti1}
For every $i\geq 1$, either $T_{\beta}^i(1)=0$ or $T_{\beta}^i(1)\not\in\mathbb{Q}$.
\end{lemma}

\begin{proof}
This follows from $T_\beta^i(1)\in\mathbb{Z}[\beta]\cap[0,1)$.
\end{proof}

\begin{lemma}\label{lem:minus1}
If $\beta$ satisfies {\rm (F)}, then $\Xi(-1)$ lies on the boundary of~$\mathcal{T}$. 
\end{lemma}

\begin{proof}
If $\beta$ satisfies~(F), then we have $d_\beta(1)=t_1\cdots t_n 0^{\omega}$, with $t_n>0$.
If $1\le j\le n$ is an integer such that $t_j>0$, then 
$$
\leftexp{\omega}{(}t_1\cdots t_{n-1}(t_n-1))t_1\cdots t_{j-1}(t_j-1)\biinf t_{j+1}\cdots t_n 0^{\omega}
$$
is a $\Xi(\beta)$-expansion of $-1$ since this sequence is admissible and 
$$
\lim_{k\to\infty} \Xi\big((t_1\cdots t_{n-1}(t_n-1))^k t_1\cdots t_j\decdot t_{j+1}\cdots t_n\big) = \lim_{k\to\infty} \Xi(\beta^{j+k n}) = 0.
$$ 
Since by definition $t_n>0$, $\Xi(-1)$ belongs to the central tile~$\mathcal{T}$.
Since $t_1>0$, $\Xi(-1)$ also lies in the tile $\mathcal{T}(y)$ where $y := \decdot t_2\cdots t_n\neq 0$.  
By Theorem~ARS, we obtain that $\Xi(-1)$ lies on the boundary of the tile~$\mathcal{T}$, as claimed.
\end{proof}

We are now ready to prove Theorem~\ref{th1}. 

\begin{proof}[Proof of Theorem~\ref{th1}] 
As previously mentioned, if $\beta$ satisfies~(F),  
we infer from \cite{AKI1} that $\gamma(\beta)>0$. 

If $\beta$ does not satisfy~(F), then Akiyama~\cite{Akiyama00} proved that the minimal polynomial $p(x)$ of $\beta$ satisfies 
either $p(0)=1$ or $p(x)=x^3-a x^2-bx-1$ with $-a+1\le b\le-2$. 
Hence, $\beta$ has either a positive real Galois conjugate or two complex Galois conjugates. 
In the first case, it is easy to see that $\gamma(\beta)=0$, as observed in~\cite{AKI1}. 
In the latter case, it is known that the origin belongs to the boundary of the central tile~$\mathcal{T}$. 
By Proposition~\ref{prop:spiral}, we get that $0$ is a spiral point with respect to~$\mathcal{T}$. 
Hence, there are rational numbers arbitrarily close to $0$ whose $\beta$-expansion is not purely periodic (as well as intervals where all rational numbers have purely periodic $\beta$-expansion). 
Consequently, we also have $\gamma(\beta)=0$ in that case, concluding the proof.
\end{proof}

\begin{proof}[Proof of Theorem~\ref{th2}] 
Let $\beta$ be a cubic Pisot unit satisfying~(F). 
Let us assume that $\mathbb{Q}(\beta)$ is not a totally real number field.    
Let us assume that $\gamma(\beta)$ is a rational number and we aim at deriving a contradiction. 
Note that by assumption we have $\Xi(-\gamma(\beta))=-\gamma(\beta)$. 

We first observe that if $\gamma(\beta)=T_{\beta}^i(1)$ for some non-negative integer~$i$, then $-\gamma(\beta)$ belongs to the boundary of the central tile~$\mathcal{T}$.  
Indeed, in that case, Lemma~\ref{lem:Ti1} implies that $\gamma(\beta)=T^0(1)=1$ and the result follows from Lemma~\ref{lem:minus1}. 

Let us assume that $-\gamma(\beta)$ belongs to the boundary of the central tile~$\mathcal{T}$. 
By Proposition~\ref{prop:spiral}, $-\gamma(\beta)$ is a spiral point 
with respect to~$\mathcal{T}$. 
Thus there exists a rational number $0<r<\gamma(\beta)$ such that $-r$ lies in the interior of a tile $\mathcal{T}(y)$ for some $y\neq 0$. 
By Theorem~ARS, $-r = \Xi(-r)$ does not belong to $\mathcal{T}$ 
and thus 
$$
(-r,r)\not\in\bigcup_{i=0}^{m+n-1}\mathcal{T}_i\times[0,T_{\beta}^i(1)).
$$ 
Theorem~IR then implies that $d_{\beta}(r)$ is not purely periodic, which contradicts the definition of $\gamma(\beta)$ since $0< r<\gamma(\beta)$.

Let us assume now that $-\gamma(\beta)$ does not belong to the boundary of the central tile~$\mathcal{T}$. 
By Corollary~\ref{cor:boundary} and our first observation, 
this ensures the existence of two integers $0\leq i\neq j \leq m+ n-1$ such that $-\gamma(\beta)\in\mathcal{T}_i\cap\mathcal{T}_j$ and $T_{\beta}^j(1)<\gamma(\beta)<T_{\beta}^i(1)$. 
By Proposition~BS, we get that $-\gamma(\beta)\in\partial \mathcal{T}_j$. 
We then infer from Proposition~\ref{prop:spiral} that $-\gamma(\beta)$ is a spiral point with respect to~$\mathcal{T}_j$. The interior of $\mathcal{T}_j$ thus contains a rational number $-r$ such that $T_{\beta}^j(1)<r<\gamma(\beta)$. 
By Proposition~BS, $-r=\Xi(-r)$ does not belong to any other subtile $\mathcal{T}_i$ and thus 
$$
(-r,r)\not\in\bigcup_{i=0}^{m+n-1}\mathcal{T}_i\times[0,T_{\beta}^i(1)).
$$ 
By Theorem~IR, the $\beta$-expansion of $r$ is not purely periodic which yields a contradiction with the fact that $0< r < \gamma(\beta)$. 
This concludes the proof. 
\end{proof}


\bibliographystyle{plain}
\bibliography{Biblio}
\end{document}